\documentclass[11pt,reqno]{amsart}
\calclayout
\usepackage{amsmath}
\usepackage{amssymb,amsthm,mathtools,mathrsfs}
\usepackage[shortlabels]{enumitem}
\usepackage[colorlinks=true,linkcolor=blue,citecolor=blue]{hyperref}

\numberwithin{equation}{section}

\DeclareMathOperator{\supp}{supp}
\DeclareMathOperator{\Ind}{Ind}
\DeclareMathOperator{\U}{U}
\DeclareMathOperator{\pr}{pr}

\newcommand{\C}{\mathbb{C}}
\newcommand{\R}{\mathbb{R}}
\newcommand{\Z}{\mathbb{Z}}
\newcommand{\cN}{\mathcal{N}}

\newcommand{\sL}{\mathscr{L}}
\newcommand{\norm}[1]{\| #1\|}

\begin{document}

\theoremstyle{plain}
\newtheorem{theorem}{Theorem}[section]
\newtheorem{lemma}[theorem]{Lemma}
\newtheorem{proposition}[theorem]{Proposition}
\newtheorem{corollary}[theorem]{Corollary}

\theoremstyle{definition}
\newtheorem{remark}[theorem]{Remark}

\title[Higher localised $\widehat{A}$-genera]{Higher localised $\widehat{A}$-genera: a note on disconnected fixed-point sets}

\author{Hao Guo}


\begin{abstract}
For a proper cocompact action of a finitely generated discrete group $\Gamma$ on a spin manifold $M$, the higher localised $\widehat{A}$-genera $\widehat{A}_g(M,\omega)$ introduced earlier by the author and Mathai were shown to be an obstruction to $\Gamma$-invariant metrics of positive scalar curvature, for $\omega$ arising from the subring of $H^*(\underline{B}\Gamma,\R)$ generated by elements of degree at most $2$, under the assumptions that the fixed-point set $M^g$ is connected and that $g$ is regular with respect to a certain real group $2$-cocycle. We show that both assumptions can be dropped. We also correct the normalisation of the phase functions used in that paper, which as stated there does not produce multipliers satisfying the normalisation condition required of them; with the corrected normalisation, all of its results hold as stated.
\end{abstract}

\maketitle

\section{Introduction}
\label{sec intro}
Let $\Gamma$ be a finitely generated discrete group acting properly on a connected spin manifold $M$, preserving the spin structure, such that $M/\Gamma$ is compact and $H^1(M)=0$. Fix a $\Gamma$-invariant Riemannian metric on $M$. For $g\in\Gamma$, let $M^g$ be the fixed-point submanifold of $g$ and $Z^g$ the centraliser of $g$ in $\Gamma$. 

Since the action of $Z^g$ on $M^g$ is proper and cocompact, there exists a \emph{cut-off function} $c^g\in C^\infty_c(M^g)$ for this action, satisfying 
$$\sum_{k\in Z^g}c^g(k^{-1}x)=1$$ for all $x\in M^g$. In \cite{GM}, V.~Mathai and the author defined, for a closed $\Gamma$-invariant differential form $\xi$ on $M$, the \emph{higher localised $\widehat{A}$-genus}
\begin{equation}
\label{eq A hat}
\widehat A_{g}(M,\xi)\coloneqq\int_{M^g}c^g\cdot\frac{\widehat A(M^g)\cdot\xi|_{M^g}}{\det(1-g e^{-R^{\mathcal{N}}/2\pi i})^{1/2}},
\end{equation}
where $R^{\cN}$ is the curvature of the Levi-Civita connection restricted to the normal bundle $\cN$ of $M^g$ in $M$. Here and below, we adopt the convention that an integral over $M^g$ means the sum over the connected components of $M^g$ of the integrals of the top-degree parts of the integrand; the components may have different dimensions, and only finitely many of them meet $\supp(c^g)$. If $M^g$ is empty, the integral is zero.

Let $f\colon M/\Gamma\to\underline{B}\Gamma$ be the classifying map for the action, where $\underline{B}\Gamma$ is the classifying space for proper $\Gamma$-actions \cite{BCH}. It was conjectured in \cite[Conjecture 1.1]{GM} that if $M$ admits a $\Gamma$-invariant metric of positive scalar curvature, then $\widehat{A}_g(M,\xi)=0$ for all $g\in\Gamma$ and all $\xi$ that are $\Gamma$-invariant lifts of forms on the orbifold $M/\Gamma$ representing classes in $f^*H^*(\underline{B}\Gamma,\R)$. Theorem 1.5 and Corollary 1.6 of \cite{GM} establish this for classes in the subring generated by $f^*H^1(\underline{B}\Gamma,\R)\cup f^*H^2(\underline{B}\Gamma,\R)$, whenever the conjugacy class $(g)$ has polynomial growth, under two additional assumptions: that $M^g$ is connected, and that $g$ is regular with respect to a real group $2$-cocycle $\alpha$ constructed from $\xi$. We now show that both assumptions can be dropped.

Let us recall where these assumptions come from. Let $\omega$ be a $\Gamma$-invariant exact $2$-form on $M$. For instance, $\omega$ could be the $\Gamma$-invariant lift of a form on $M/\Gamma$ representing a class in $f^*H^2(\underline{B}\Gamma,\R)$, which is exact by \cite[subsection 2.2]{GM}. Choose a one-form $\eta$ on $M$ with $d\eta=\omega$. For each $\gamma\in\Gamma$, the one-form $\gamma^*\eta-\eta$ is closed, hence exact since $H^1(M)=0$, so there exist smooth functions $\phi_\gamma$ on $M$ such that
\begin{equation}
\label{eq derivative of phi}
\gamma^*\eta-\eta=d\phi_\gamma.
\end{equation}
Since $M$ is connected, each $\phi_\gamma$ is determined up to an additive constant. Fix $x_0\in M$ and choose the constants so that
\begin{equation}
\label{eq normalisation}
\phi_\gamma(x_0)+\phi_\gamma(\gamma^{-1}x_0)=0
\end{equation}
for all $\gamma\in\Gamma$; for each $\gamma$ there is exactly one constant that achieves this. This normalisation differs from the one in \cite[(2.4)]{GM}, for a reason explained in Remark \ref{rem normalisation} below. Then
\begin{equation}
\label{eq alpha}
\alpha(\gamma,\gamma')=\frac{1}{2\pi}\big(\phi_{\gamma'}(x_0)+\phi_\gamma(\gamma'x_0)-\phi_{\gamma\gamma'}(x_0)\big)
\end{equation}
is a real group $2$-cocycle on $\Gamma$. From this we get an associated family 
$$\sigma^s\coloneqq e^{2\pi is\alpha},$$ 
for $s\in\R$, of multipliers in the sense of \cite[Definition 2.1]{GM}; see Remark \ref{rem normalisation}. Recall that an element $g\in\Gamma$ is called \emph{$\alpha$-regular} if
\begin{equation}
\label{eq alpha regular}
\alpha(g,z)=\alpha(z,g)
\end{equation}
for all $z\in Z^g$. This is the condition under which the weighted traces $\tau^{(g)}_{\sigma^s}$ on the twisted group algebras $\C^{\sigma^s}\Gamma$ are defined \cite[Corollary 3.11]{GM}. For $\alpha$-regular $g$ with $(g)$ of polynomial growth, the fixed-point formula \cite[Theorem 1.5(i)]{GM} computes the pairing of $\tau^{(g)}_{\sigma^s}$ with the higher index of the projectively invariant Dirac operator $D^s$ twisted by the connection $d+is\eta$ as an integral of the form \eqref{eq A hat}, but with an additional factor of $e^{-is\phi_g}$ in the integrand.

\begin{remark}
\label{rem normalisation}
The normalisation adopted in \cite[(2.4)]{GM} is $\phi_\gamma(\gamma^{-1}x_0)=0$. It gives $\phi_e\equiv 0$ and $\alpha(e,\gamma)=\alpha(\gamma,e)=0$, but not $\alpha(\gamma,\gamma^{-1})=0$, since by \eqref{eq alpha} one has $2\pi\alpha(\gamma,\gamma^{-1})=\phi_{\gamma^{-1}}(x_0)$. Thus, contrary to the assertion following \cite[(2.6)]{GM}, $\sigma^s$ need not satisfy condition (ii) of \cite[Definition 2.1]{GM}. We correct this presently with \eqref{eq normalisation}, after which all of the results in \cite{GM} go through. (Also, in \cite[(4.20)]{GM} and the display following it, the term $\phi_k(k^{-1}x)$ should carry a minus sign, as in \cite[(4.18)]{GM}.)
\end{remark}

Restricting \eqref{eq derivative of phi} to $M^g$ shows that $\phi_g$ is locally constant on $M^g$; see \cite[Remark 2.2]{GM}. If $M^g$ is connected, one may take $x_0\in M^g$, in which case $\phi_g$ vanishes identically on $M^g$ and the phase disappears. This is the place where the connectedness of $M^g$ was used in \cite{GM}.

The starting point of this note is the elementary observation that $\phi_g$ is in fact constant on all of $M^g$, whether or not $M^g$ is connected. To see this, note that since the action is proper, an element with a fixed point has finite order, say $r$. Summing \eqref{eq derivative of phi} over the powers of $g$ gives
$$d\Big(\sum_{j=0}^{r-1}\phi_g\circ g^j\Big)=(g^r)^*\eta-\eta=0.$$
The sum is therefore constant on $M$, since $M$ is connected, while on $M^g$ it equals $r\phi_g$. This is Lemma \ref{lem constant phase} below. We will see in Lemma \ref{lem regular} that a similar argument, applied to the central extension of $\Gamma$ defined by $\alpha$, shows that elements of finite order are automatically $\alpha$-regular. Granted these two facts, the proof of \cite[Theorem 1.5(ii)]{GM} goes through with the common value of $\phi_g$ on $M^g$ in place of its value on the single component; we carry this out in section \ref{sec proof}, and in section \ref{sec products} we treat the products of $2$-forms needed for \cite[Corollary 1.6]{GM}. The results are as follows.
\begin{theorem}
\label{thm main}
Let $\Gamma$ and $M$ be as above, and let $\omega$ be a $\Gamma$-invariant exact $2$-form on $M$. Let $g\in\Gamma$ be an element whose conjugacy class $(g)$ has polynomial growth. If $M$ admits a $\Gamma$-invariant Riemannian metric of positive scalar curvature, then
$$\widehat{A}_g(M,\omega^k)=0$$
for each $k\geq 0$.
\end{theorem}
\begin{corollary}
\label{cor main}
Let $\Gamma$ and $M$ be as above, and let $g\in\Gamma$ be an element whose conjugacy class $(g)$ has polynomial growth. Then \cite[Conjecture 1.1]{GM} holds for $g$ and all $\xi$ arising from the subring of $H^*(\underline{B}\Gamma,\R)$ generated by $f^*H^1(\underline{B}\Gamma,\R)\cup f^*H^2(\underline{B}\Gamma,\R)$.
\end{corollary}
In other words, the conclusions of \cite[Theorem 1.5(ii)]{GM} and \cite[Corollary 1.6]{GM} hold without assuming that $M^g$ is connected or that $g$ is $\alpha$-regular. The remaining hypotheses of \cite{GM}, namely that $M$ is connected, $H^1(M)=0$, $M/\Gamma$ is compact and $(g)$ has polynomial growth, are unchanged. The same remark applies to the index formula \cite[Theorem 1.5(i)]{GM} and to the hypersurface obstruction \cite[Theorem 1.7]{GM}; see Remarks \ref{rem index formula} and \ref{rem hypersurface}. Remark \ref{rem normalisation} corrects the normalisation of the functions $\phi_\gamma$ used in \cite{GM}; this correction is needed for the multipliers of \cite{GM} to satisfy the normalisation condition imposed on them there, and it does not affect any of the results of \cite{GM}.

\hfill\vskip 0.3in
\section{Elements of finite order}
\label{sec finite order}
\begin{lemma}
\label{lem finite order}
Let a discrete group $\Gamma$ act properly on a manifold $M$, and let $g\in\Gamma$. If $M^g\neq\emptyset$, then $g$ has finite order.
\end{lemma}
\begin{proof}
If $gx=x$, then $g$ lies in the stabiliser $\Gamma_x$ of $x$, which is compact since the action is proper, hence finite since $\Gamma$ is discrete.
\end{proof}

\begin{lemma}
\label{lem constant phase}
Let $M$ be connected, let $g\in\Gamma$ have finite order $r$, and let $\eta$ and $\phi_g$ be as in \eqref{eq derivative of phi}. Then $\phi_g$ is constant on $M^g$. If the base point $x_0$ in \eqref{eq normalisation} lies in $M^g$, then $\phi_g$ vanishes on $M^g$.
\end{lemma}
\begin{proof}
Consider the smooth function
$$S_g\coloneqq\sum_{j=0}^{r-1}(g^j)^*\phi_g=\sum_{j=0}^{r-1}\phi_g\circ g^j$$
on $M$. Since $(g^j)^*g^*\eta=(g^{j+1})^*\eta$, equation \eqref{eq derivative of phi} gives
$$dS_g=\sum_{j=0}^{r-1}(g^j)^*(g^*\eta-\eta)=\sum_{j=0}^{r-1}\big((g^{j+1})^*\eta-(g^j)^*\eta\big)=(g^r)^*\eta-\eta=0,$$
so $S_g$ is equal to a constant $C$ on $M$. If $x\in M^g$, then $g^jx=x$ for all $j$, so $S_g(x)=r\phi_g(x)$. Hence $\phi_g=C/r$ on $M^g$. If $x_0\in M^g$, then $g^{-1}x_0=x_0$ and \eqref{eq normalisation} gives $2\phi_g(x_0)=0$, so $C=0$.
\end{proof}

\begin{lemma}
\label{lem regular}
Let $\alpha\in Z^2(\Gamma,\R)$. Then every element of $\Gamma$ of finite order is $\alpha$-regular. In particular, such an element is $\sigma^s$-regular for the multiplier $\sigma^s=e^{2\pi is\alpha}$, for every $s\in\R$.
\end{lemma}
\begin{proof}
Let $g\in\Gamma$ have order $r$, and let $z\in Z^g$. If $\beta\colon\Gamma\to\R$, then $(\delta\beta)(\gamma,\gamma')=\beta(\gamma)+\beta(\gamma')-\beta(\gamma\gamma')$ is symmetric in $\gamma$ and $\gamma'$ whenever these commute, so the quantity $\alpha(g,z)-\alpha(z,g)$ depends only on the class of $\alpha$ in $H^2(\Gamma,\R)$. We may therefore assume that $\alpha(e,\gamma)=\alpha(\gamma,e)=0$ for all $\gamma$. (The cocycle \eqref{eq alpha} already has this property, since $\phi_e\equiv 0$.)

Let $E=\R\times_\alpha\Gamma$ be the central extension of $\Gamma$ by $\R$ defined by $\alpha$, with product
$$(u,\gamma)(v,\gamma')=(u+v+\alpha(\gamma,\gamma'),\gamma\gamma').$$
Put $A=(0,g)$ and $B=(0,z)$. Then $AB=(\alpha(g,z),gz)$ and $BA=(\alpha(z,g),gz)$, so
$$AB=(c,e)BA,\qquad c\coloneqq\alpha(g,z)-\alpha(z,g),$$
where $(c,e)$ lies in the central subgroup $\R\times\{e\}$. By induction, $A^nB=(nc,e)BA^n$ for all $n\geq 1$. Now $A^r$ lies in $\R\times\{e\}$, since $g^r=e$, and is therefore central. Taking $n=r$ gives $(rc,e)=(0,e)$, so $rc=0$ and hence $c=0$. The last statement follows by exponentiating \eqref{eq alpha regular}.
\end{proof}

\begin{remark}
The proof relies on the fact that $\R$ has no torsion. For a general multiplier $\sigma\in Z^2(\Gamma,\U(1))$, elements of finite order need not be $\sigma$-regular: if $\Gamma=\Z/2\times\Z/2$ and $\sigma$ represents the non-trivial class in $H^2(\Gamma,\U(1))$, then $\C^\sigma\Gamma$ is the quaternion algebra, in which the generators anticommute. The multipliers appearing in \cite{GM} are, however, all of the form $e^{2\pi is\alpha}$ with $\alpha$ real-valued, so Lemma \ref{lem regular} applies.
\end{remark}

\hfill\vskip 0.3in
\section{Proof of Theorem \ref{thm main}}
\label{sec proof}
We will use the following observation about the dimensions of the components of $M^g$. Since $g$ preserves the spin structure, it preserves the orientation of $M$. At a point $x\in M^g$, $g$ acts on $T_xM=T_xM^g\oplus\cN_x$ as the identity on the first summand and as an orthogonal transformation without eigenvalue $1$ on the second. The non-real eigenvalues of the latter come in conjugate pairs, and its $(-1)$-eigenspace is even-dimensional because its determinant is $1$. Hence the rank of $\cN$ is even on every component of $M^g$, and each component of $M^g$ has dimension congruent to $\dim M$ modulo $2$.

\begin{proof}[Proof of Theorem \ref{thm main}]
If $M^g=\emptyset$ there is nothing to prove, so assume $M^g\neq\emptyset$. Then $g$ has finite order by Lemma \ref{lem finite order}, and is $\alpha$-regular by Lemma \ref{lem regular}.

Suppose first that $\dim M$ is odd. Then every component of $M^g$ is odd-dimensional, while $\widehat A(M^g)$, the determinant factor in \eqref{eq A hat}, and $\omega^k$ all have even degree. Hence the integrand in \eqref{eq A hat} has no top-degree part on any component, so $\widehat A_g(M,\omega^k)=0$.

Next, suppose $\dim M$ is even. Fix a $\Gamma$-invariant metric of positive scalar curvature on $M$. For $s\in\R$, let $D^s$ be the $(\Gamma,\sigma^s)$-invariant Dirac operator of \cite[Definition 2.4]{GM}. In other words, it is the spin-Dirac operator twisted by the connection 
$$\nabla^{\sL,s}=d+is\eta$$ on a trivial line bundle $\sL$. By Remark \ref{rem normalisation}, Lemma \ref{lem regular} and \cite[Corollary 3.11]{GM}, the trace $\tau^{(g)}_{\sigma^s}$ is defined for every $s$, and since $(g)$ has polynomial growth, it induces a map $(\tau^{(g)}_{\sigma^s})_*$ on $K_0(C^*_r(\Gamma,\sigma^s))$. The proof of \cite[Theorem 1.5(i)]{GM} (in particular, see \cite[Theorem 4.5 and (4.23)]{GM}) then gives, for every $s\in\R$,
\begin{equation}
\label{eq fixed point formula}
\big(\tau^{(g)}_{\sigma^s}\big)_*\Ind_{\Gamma,\sigma^s}(D^s)=\int_{M^g}e^{-is\phi_g}c^g\cdot\frac{\widehat A(M^g)\cdot e^{-s\omega/2\pi i}|_{M^g}}{\det(1-g e^{-R^{\mathcal{N}}/2\pi i})^{1/2}}.
\end{equation}
(Note that the normalisation $-s\omega/2\pi i$ of the exponent, which depends on the convention for the Chern character, plays no role below; what matters is that the coefficient of $s\omega$ is a non-zero constant.) By \cite[Corollary 4.8]{GM}, there exists $\delta>0$ such that $\Ind_{\Gamma,\sigma^s}(D^s)=0$ for all $s\in(0,\delta)$. By Lemma \ref{lem constant phase}, there is a constant $a_g\in\R$ such that $\phi_g\equiv a_g$ on $M^g$. Expanding the exponential, \eqref{eq fixed point formula} gives
\begin{equation}
\label{eq polynomial}
0=e^{-isa_g}\sum_{k=0}^{\dim M/2}\frac{1}{k!}\Big(\frac{-s}{2\pi i}\Big)^k\,\widehat{A}_g(M,\omega^k)
\end{equation}
for all $s\in(0,\delta)$. This sum is a polynomial in $s$ that vanishes on an interval, hence all of its coefficients are zero. The result follows.
\end{proof}

\begin{remark}
\label{rem index formula}
Lemma \ref{lem constant phase} also simplifies the index formula \cite[Theorem 1.5(i)]{GM}. In the notation used there, the values $\phi_g(x_j)$ on the components $M^g_j$ of $M^g$ all coincide with $a_g$, so that
\begin{align*}
\big(\tau^{(g)}_{\sigma}\big)_*\Ind_{\Gamma,\sigma}(D)&=e^{-ia_g}\int_{M^g}c^g\cdot\frac{\widehat A(M^g)\cdot e^{-\omega/2\pi i}|_{M^g}}{\det(1-g e^{-R^{\mathcal{N}}/2\pi i})^{1/2}}\\
&=e^{-ia_g}\sum_{k\geq 0}\frac{1}{k!}\Big(\frac{-1}{2\pi i}\Big)^k\,\widehat{A}_g(M,\omega^k),
\end{align*}
and $a_g=0$ if $x_0$ is chosen to lie in $M^g$. Moreover, by Lemma \ref{lem regular}, the hypothesis in \cite[Theorem 1.5]{GM} that $g$ be $\alpha$-regular is automatically satisfied whenever $M^g\neq\emptyset$.
\end{remark}

\begin{remark}
\label{rem exponential polynomials}
Lemma \ref{lem constant phase} is not strictly necessary for Theorem \ref{thm main}. Without it, \eqref{eq fixed point formula} and the vanishing of the index give $\sum_j e^{-isa_j}P_j(s)=0$ for $s\in(0,\delta)$, where the sum is over the components $M^g_j$ of $M^g$ meeting $\supp(c^g)$, $a_j$ is the value of $\phi_g$ on $M^g_j$, and $P_j$ is a polynomial. Since the functions $s^ke^{-isa}$, for distinct pairs $(k,a)$, are linearly independent on any interval, we have $\sum_{j\colon a_j=a}P_j=0$ for each value $a$. Summing over $a$ gives $\sum_j P_j=0$, which is \eqref{eq polynomial} with the phase removed. We have given the argument via Lemma \ref{lem constant phase} because it yields the stronger statement that the phases on the components agree, and this is what is needed in Remark \ref{rem index formula}.
\end{remark}

\begin{remark}
\label{rem hypersurface}
The same applies to \cite[Theorem 1.7]{GM}, where $H\subseteq M$ is a connected $\Gamma$-cocompact hypersurface with trivial normal bundle, $M$ carries a complete $\Gamma$-invariant metric with non-negative scalar curvature that is positive near $H$, and $H^g$ is assumed connected. There the functions $\phi_\gamma$ used on $H$ are the restrictions of the functions \eqref{eq derivative of phi} on $M$ (see \cite[subsection 5.2]{GM}), and the connectedness of $H^g$ enters only through the requirement that $\phi_g$ be constant on $H^g$. Since $H^g\subseteq M^g$ and $M$ is connected, Lemma \ref{lem constant phase} implies this without any assumption on $H^g$. Thus the conclusion $\widehat A_g(H,\omega^k|_H)=0$ for all $k\geq 0$ holds whenever $(g)$ has polynomial growth.
\end{remark}

The phase in \eqref{eq fixed point formula} can also be removed by changing the primitive $\eta$ rather than the base point $x_0$.
\begin{proposition}
\label{prop averaging}
Let $g\in\Gamma$ have finite order $r$, and let $\omega=d\eta$ be $g$-invariant. Then
$$\bar\eta\coloneqq\frac{1}{r}\sum_{j=0}^{r-1}(g^j)^*\eta$$
satisfies $d\bar\eta=\omega$ and $g^*\bar\eta=\bar\eta$. If $H^1(M)=0$, then $\bar\eta-\eta=df$ for a smooth function $f$ on $M$, and the functions
$$\bar\phi_\gamma\coloneqq\phi_\gamma+\gamma^*f-f+b_\gamma,\qquad\gamma\in\Gamma,$$
with $b_\gamma\in\R$ constants, satisfy $\gamma^*\bar\eta-\bar\eta=d\bar\phi_\gamma$. The constants $b_\gamma$ can be chosen so that the $\bar\phi_\gamma$ are normalised as in \eqref{eq normalisation} and $\bar\phi_g\equiv 0$, and the resulting cocycle $\bar\alpha$ differs from $\alpha$ by the coboundary of $\gamma\mapsto b_\gamma/2\pi$.
\end{proposition}
\begin{proof}
Since $\omega$ is $g$-invariant, $d\bar\eta=\frac{1}{r}\sum_j(g^j)^*\omega=\omega$, and since $g^r=e$, reindexing the sum gives $g^*\bar\eta=\bar\eta$. The one-form $\bar\eta-\eta$ is closed, hence exact. Next, $d(\gamma^*f-f)=\gamma^*(\bar\eta-\eta)-(\bar\eta-\eta)$, so $d\bar\phi_\gamma=\gamma^*\bar\eta-\bar\eta$ for any constants $b_\gamma$. As $d\bar\phi_g=g^*\bar\eta-\bar\eta=0$, the function $\bar\phi_g$ is constant, and we may take $b_g$ so that it is zero; this is compatible with \eqref{eq normalisation}. Finally, the expression $\phi_{\gamma'}+\gamma'^*\phi_\gamma-\phi_{\gamma\gamma'}$ is constant on $M$ by \cite[(2.3)]{GM}, and its value at $x_0$ is $2\pi\alpha(\gamma,\gamma')$ by \eqref{eq alpha}. Substituting $\bar\phi$ for $\phi$, the terms involving $f$ cancel, and what remains is $2\pi\alpha(\gamma,\gamma')+b_{\gamma'}+b_\gamma-b_{\gamma\gamma'}$.
\end{proof}
Going through the proof of Theorem \ref{thm main} with $\bar\eta$ in place of $\eta$, we see that the connection $d+is\bar\eta$ has the same curvature $is\omega$, so \cite[Corollary 4.8]{GM} applies as before, while the phase in \eqref{eq fixed point formula} is absent. Since $\alpha$ and $\bar\alpha$ are cohomologous, $g$ is $\bar\alpha$-regular if and only if it is $\alpha$-regular, which agrees with Lemma \ref{lem regular}.

\hfill\vskip 0.3in
\section{Products and the proof of Corollary \ref{cor main}}
\label{sec products}
\begin{theorem}
\label{thm products}
Let $\Gamma$, $M$ and $g$ be as in Theorem \ref{thm main}. Let $\omega_1,\ldots,\omega_m$ be $\Gamma$-invariant exact $2$-forms on $M$. If $M$ admits a $\Gamma$-invariant Riemannian metric of positive scalar curvature, then
$$\widehat A_g(M,\omega_1\wedge\cdots\wedge\omega_m)=0.$$
\end{theorem}
\begin{proof}
As in the proof of Theorem \ref{thm main}, we may assume that $M^g\neq\emptyset$ and that $\dim M$ is even, so that $g$ has finite order. For each $j$, choose a one-form $\eta_j$ with $d\eta_j=\omega_j$ and functions $\phi_{j,\gamma}$ with $\gamma^*\eta_j-\eta_j=d\phi_{j,\gamma}$, normalised as in \eqref{eq normalisation} with respect to a common base point $x_0$. For $\mathbf{s}=(s_1,\ldots,s_m)\in\R^m$, put
$$\eta_{\mathbf{s}}=\sum_{j=1}^m s_j\eta_j,\qquad\omega_{\mathbf{s}}=d\eta_{\mathbf{s}}=\sum_{j=1}^m s_j\omega_j,\qquad\phi_{\mathbf{s},\gamma}=\sum_{j=1}^m s_j\phi_{j,\gamma}.$$
These satisfy \eqref{eq derivative of phi} and \eqref{eq normalisation} for $\omega_{\mathbf{s}}$, and the associated cocycle is $\alpha_{\mathbf{s}}=\sum_j s_j\alpha_j$, with respect to which $g$ is regular by Lemma \ref{lem regular}. Let $D^{\mathbf{s}}$ be the spin-Dirac operator twisted by $d+i\eta_{\mathbf{s}}$, which is $(\Gamma,\sigma_{\mathbf{s}})$-invariant for $\sigma_{\mathbf{s}}=e^{2\pi i\alpha_{\mathbf{s}}}$. By Lemma \ref{lem constant phase}, 
$$\phi_{\mathbf{s},g}\equiv\sum_j s_ja_j$$ on $M^g$, where $a_j$ is the value of $\phi_{j,g}$ on $M^g$. Thus \eqref{eq fixed point formula}, applied to $\omega_{\mathbf{s}}$ with $s=1$, gives
$$\big(\tau^{(g)}_{\sigma_{\mathbf{s}}}\big)_*\Ind_{\Gamma,\sigma_{\mathbf{s}}}(D^{\mathbf{s}})=e^{-i\sum_j s_ja_j}\int_{M^g}c^g\cdot\frac{\widehat A(M^g)\cdot e^{-\omega_{\mathbf{s}}/2\pi i}|_{M^g}}{\det(1-g e^{-R^{\mathcal{N}}/2\pi i})^{1/2}}.$$
The curvature of $d+i\eta_{\mathbf{s}}$ is $i\omega_{\mathbf{s}}$, and $\norm{c(\omega_{\mathbf{s}})}_x\leq\sum_j|s_j|\sup_{y\in M}\norm{c(\omega_j)}_y$, where the suprema are finite since the $\omega_j$ are $\Gamma$-invariant and $M/\Gamma$ is compact. Hence, by \cite[Proposition 4.6]{GM}, if the metric has positive scalar curvature then $\Ind_{\Gamma,\sigma_{\mathbf{s}}}(D^{\mathbf{s}})=0$ for all $\mathbf{s}$ in a neighbourhood $V$ of $0$ in $\R^m$. It follows that
\begin{align*}
P(\mathbf{s})&\coloneqq\int_{M^g}c^g\cdot\frac{\widehat A(M^g)\cdot e^{-\omega_{\mathbf{s}}/2\pi i}|_{M^g}}{\det(1-g e^{-R^{\mathcal{N}}/2\pi i})^{1/2}}\\
&=\sum_{k=0}^{\dim M/2}\frac{1}{k!}\Big(\frac{-1}{2\pi i}\Big)^k\int_{M^g}c^g\cdot\frac{\widehat A(M^g)\cdot\omega_{\mathbf{s}}^k|_{M^g}}{\det(1-g e^{-R^{\mathcal{N}}/2\pi i})^{1/2}}
\end{align*}
vanishes on $V$. Being a polynomial in $s_1,\ldots,s_m$, it vanishes identically on $\R^m$. The monomial $s_1s_2\cdots s_m$ occurs only in the term $k=m$, and since $2$-forms commute, it occurs in $\omega_{\mathbf{s}}^m=(\sum_j s_j\omega_j)^m$ with coefficient $m!\,\omega_1\wedge\cdots\wedge\omega_m$. Its coefficient in $P$ is $$(-1/2\pi i)^m\,\widehat A_g(M,\omega_1\wedge\cdots\wedge\omega_m),$$ 
which must therefore be zero.
\end{proof}

\begin{proof}[Proof of Corollary \ref{cor main}]
Let $\xi$ be the $\Gamma$-invariant lift of a closed form on $M/\Gamma$ representing a class in the subring of $H^*(\underline{B}\Gamma,\R)$ generated by $f^*H^1(\underline{B}\Gamma,\R)\cup f^*H^2(\underline{B}\Gamma,\R)$. Note first that $\widehat A_g(M,\xi)$ depends only on the class of $\xi$ in the orbifold de Rham cohomology of $M/\Gamma$. Indeed, if $\xi=d\theta$ with $\theta$ $\Gamma$-invariant, then, as the characteristic form in \eqref{eq A hat} is closed and $Z^g$-invariant, the integrand in \eqref{eq A hat} is $c^g$ times an exact $Z^g$-invariant form $d\beta$ on $M^g$, and $\int_{M^g}c^g\,d\beta=0$, being the integral of $d\beta$ over the compact orbifold $M^g/Z^g$. By linearity, we may therefore assume that
$$\xi=a_1\wedge\cdots\wedge a_p\wedge\omega_1\wedge\cdots\wedge\omega_q,$$
where the $a_i$ are $\Gamma$-invariant lifts of closed one-forms on $M/\Gamma$ representing classes in $f^*H^1(\underline{B}\Gamma,\R)$, and the $\omega_j$ are $\Gamma$-invariant lifts of closed two-forms representing classes in $f^*H^2(\underline{B}\Gamma,\R)$. The $\omega_j$ are exact by \cite[subsection 2.2]{GM}, and the $a_i$ are exact since $H^1(M)=0$. Recall also from section \ref{sec proof} that each component of $M^g$ has dimension congruent to $\dim M$ modulo $2$.

Suppose first that $\dim M$ is even. If $p$ is odd, then $\xi$ has odd degree, so the integrand in \eqref{eq A hat} has no top-degree part on any component of $M^g$, and $\widehat A_g(M,\xi)=0$. If $p$ is even, write $a_1=du_1$. Then $a_1\wedge a_2=d(u_1a_2)$ is a $\Gamma$-invariant exact $2$-form, and similarly for $a_3\wedge a_4$, and so on. Thus $\xi$ is a product of $p/2+q$ invariant exact $2$-forms, and $\widehat A_g(M,\xi)=0$ by Theorem \ref{thm products}.

Now suppose $\dim M$ is odd. If $\xi$ has even degree, then $\widehat A_g(M,\xi)=0$ by the same parity argument. If $\xi$ has odd degree, we argue by suspension. Let 
$$\widetilde M=M\times\R$$ with the product spin structure and the action of $\widetilde\Gamma=\Gamma\times\Z$ given by 
$$(\gamma,n)(x,t)=(\gamma x,t+n).$$
We now make a number of simple observations. First observe that the action of $\widetilde\Gamma$ is proper and cocompact, and that $\widetilde M$ is connected with $H^1(\widetilde M)=H^1(M)=0$. Note also that the product of a $\Gamma$-invariant metric of positive scalar curvature on $M$ with the standard metric on $\R$ is a $\widetilde\Gamma$-invariant metric of positive scalar curvature on $\widetilde M$. Let $\tilde g=(g,0)$. Then 
$$\widetilde M^{\tilde g}=M^g\times\R.$$ Note that the normal bundle of $\widetilde M^{\tilde g}$ in $\widetilde M$ is the pullback of $\cN$ with the same action of $g$, and $Z^{\tilde g}=Z^g\times\Z$. The conjugacy class $(\tilde g)=(g)\times\{0\}$ has polynomial growth with respect to the word length defined by a generating set of $\Gamma$ together with $1\in\Z$. Since $\underline{E}\Gamma\times\R$ is a model for $\underline{E}\widetilde\Gamma$, the classifying map of $\widetilde M/\widetilde\Gamma=M/\Gamma\times\R/\Z$ is $f\times\mathrm{id}$, and the form
$$\tilde\xi\coloneqq\pr_M^*\xi\wedge dt$$
is the $\widetilde\Gamma$-invariant lift of a form on $\widetilde M/\widetilde\Gamma$ representing a class in the subring of $H^*(\underline{B}\widetilde\Gamma,\R)$ generated by elements of degree at most $2$, as $dt$ descends to a generator of $H^1(\R/\Z,\R)$. Choose $\chi\in C^\infty_c(\R)$ with $\sum_{n\in\Z}\chi(t-n)=1$, so that $\int_\R\chi\,dt=1$. Then $c^g(x)\chi(t)$ is a cut-off function for the action of $Z^{\tilde g}$ on $\widetilde M^{\tilde g}$, and
$$\widehat A_{\tilde g}(\widetilde M,\tilde\xi)=\pm\int_{M^g}c^g\cdot\frac{\widehat A(M^g)\cdot\xi|_{M^g}}{\det(1-g e^{-R^{\mathcal{N}}/2\pi i})^{1/2}}\int_\R\chi\,dt=\pm\widehat A_g(M,\xi).$$
Since $\widetilde M$ is even-dimensional, the left-hand side vanishes by the first part of the proof.
\end{proof}

\hfill\vskip 0.3in
\section*{Acknowledgement}
The author is grateful to Charli Wallace for pointing out that the normalisation \cite[(2.4)]{GM} does not yield condition (ii) of \cite[Definition 2.1]{GM}, which led to Remark \ref{rem normalisation}.

\hfill\vskip 0.3in
\bibliographystyle{plain}
\bibliography{refs}

\end{document}